\documentclass[reqno]{amsart}
\usepackage{amsmath,amsthm,amssymb,mathrsfs,stmaryrd,eucal,mathbbol}
\usepackage[all,cmtip]{xy}
\usepackage{hyperref}
\usepackage{tikz,capt-of}

\numberwithin{equation}{section}

\newcommand\void[1]{}

\newcommand{\C}{\mathbb{C}}

\newcommand{\CA}{\mathcal{A}}

\newcommand{\CC}{\mathcal{C}}
\newcommand{\CD}{\mathcal{D}}
\newcommand{\CE}{\mathcal{E}}
\newcommand{\CF}{\mathcal{F}}
\newcommand{\CG}{\mathcal{G}}

\newcommand{\CL}{\mathcal{L}}
\newcommand{\CM}{\mathcal{M}}
\newcommand{\CN}{\mathcal{N}}

\newcommand{\CZ}{\mathcal{Z}}

\newcommand{\FC}{\mathfrak{C}}
\newcommand{\FD}{\mathfrak{D}}
\newcommand{\FE}{\mathfrak{E}}
\newcommand{\FF}{\mathfrak{F}}
\newcommand{\FG}{\mathfrak{G}}

\newcommand{\FL}{\mathfrak{L}}
\newcommand{\FM}{\mathfrak{M}}
\newcommand{\FN}{\mathfrak{N}}

\newcommand{\bk}{\mathbf{H}}

 \DeclareMathOperator{\Hom}{Hom}

 \DeclareMathOperator{\Id}{Id}

 \DeclareMathOperator{\Fun}{Fun}

\newcommand{\rev}{\mathrm{rev}}
\newcommand{\op}{\mathrm{op}}
\newcommand{\one}{\mathbf1}

\newtheorem{thm}{Theorem}[section]

\newtheorem{prop}[thm]{Proposition}
\newtheorem{cor}[thm]{Corollary}

\newtheorem{prop-defn}[thm]{Proposition-Definition}

\theoremstyle{definition}
\newtheorem{defn}[thm]{Definition}

\newtheorem{exam}[thm]{Example}
\newtheorem{rem}[thm]{Remark}

\theoremstyle{remark}

\begin{document}

\title[Extended TQFT]{Extended TQFT arising from enriched multi-fusion categories}
\author{Hao Zheng}
\address{Department of Mathematics, Peking University, Beijing 100871, China}
\email{hzheng@math.pku.edu.cn}
\maketitle

\begin{abstract}
We define a symmetric monoidal (4,3)-category with duals whose objects are certain enriched multi-fusion categories.
For every modular tensor category $\CC$, there is a self enriched multi-fusion category $\FC$ giving rise to an object of this symmetric monoidal (4,3)-category.
%For every modular tensor category $\CC$, it contains an object $\FC$ canonically associated to $\CC$.
We conjecture that the extended 3D TQFT given by the fully dualizable object $\FC$  extends the 1-2-3-dimensional Reshetikhin-Turaev TQFT associated to the modular tensor category $\CC$ down to dimension zero.
\end{abstract}

\section{Introduction}

In his seminal paper \cite{W2}, Witten gave an explanation of the Jones polynomial \cite{J1,J2} in terms of 3D Chern-Simons theory and generalized the Jones polynomial to invariants of 3-manifolds (with ribbon links inside). The corresponding mathematical theory is the Reshetikhin-Turaev invariant \cite{RT1,RT2} which is defined for any modular tensor category $\CC$ including those arising from Chern-Simons theory. Unlike many other topological invariants, the Reshetikhin-Turaev invariant essentially involves a framing anomaly which has already been observed in Witten's work.
But rather, the double theory of the Reshetikhin-Turaev invariant does not suffer from such an anomaly and, in fact, coincides with the Turaev-Viro invariant of 3-manifolds associated to the same modular tensor category $\CC$ \cite{TV,Tu}. Moreover, the Reshetikhin-Turaev invariant turns out to be a boundary of a 4D theory \cite{BGM}, the Crane-Yetter invariant of 4-manifolds \cite{CY,CKY} associated to the modular tensor category $\CC$.

\smallskip

These invariants are naturally settled in the framework of TQFT introduced in \cite{W1,At} and extended TQFT developed in, for example, \cite{La,Fr1,BD,Lu}. The Reshetikhin-Turaev invariant is incorporated in a beautiful 1-2-3-dimensional extended TQFT $\CZ^{RT}_\CC$ where the 1-2-dimensional theory is formulated in terms of modular functor and Moore-Seiberg data (see \cite{Tu,BK} and references therein, see also \cite{Fr3}). If $\CC$ arises from Chern-Simons theory with finite gauge group (known as Dijkgraaf-Witten theory), the theory can be extended even down to dimensional zero \cite{Fr2}.

Extended TQFT for Turaev-Viro invariant and Crane-Yetter invariant were built from various point of view (see, for example, \cite{BaK,dss1,Fr3}).
%, partly because there is no framing anomaly involved.
In fact, both invariants could be extended all the way down to dimension zero. According to the cobordism hypothesis, which is proposed by Baez and Dolan \cite{BD} and proved in sketch by Lurie \cite{Lu} (see also \cite{Fr4}), a fully extended TQFT is determined by its value on a point. By realizing the modular tensor category $\CC$ as a fully dualizable object in a certain symmetric monoidal 3-category, one obtains a fully extended 3D TQFT $\CZ^{TV}_\CC$ whose values on 3-manifolds give (modulo certain reasonable conjectures) the Turaev-Viro invariant \cite{dss1}. Similarly, by realizing the modular tensor category $\CC$ as a fully dualizable object in a certain symmetric monoidal 4-category, one obtains a fully extended 4D TQFT $\CA_\CC$ whose values on 4-manifolds should recover the Crane-Yetter invariant \cite{Fr3}. As before, the double theory of $\CZ^{RT}_\CC$ agrees with $\CZ^{TV}_\CC$ \cite{Ba2}, and $\CZ^{RT}_\CC$ is a boundary theory of $\CA_\CC$ \cite{Fr3}. Although there still remain substantial works to be fulfilled to make many of these statements rigorous, the global picture has been fairly clear.

\smallskip

It is natural to ask whether the Reshetikhin-Turaev TQFT $\CZ^{RT}_\CC$ can be extended further to dimensional zero. According to the cobordism hypothesis, this is equivalent to ask what mathematical object should be assigned to a point. Experiences on TQFT told us this amounts to find a mathematical object whose ``center'' is the modular tensor category $\CC$.
For example, if $\CC$ is the Drinfeld center of a spherical fusion category $\CD$, such an extension is available. In fact, the Turaev-Viro invariant is also defined for a spherical fusion category \cite{Oc,BW} and it is known that $\CZ^{RT}_{Z(\CD)}$ is equivalent to $\CZ^{TV}_\CD$ \cite{Ba1,Ba2,TVi}.
But for a general modular tensor category $\CC$, this question remains open. See \cite{FHLT} for a treatment where $\CC$ arises from Chern-Simons theory with torus gauge group.
Recently, Henriques also proposed a candidate by showing that certain unitary modular tensor categories (completed by separable Hilbert spaces) are the Drinfeld centers of certain categories of solitons \cite{He1,He2}.

\smallskip

In this paper, we propose another approach to this problem. We introduce the notions of enriched multi-fusion categories (see \cite{MP} for a definition of monoidal category enriched in a braided monoidal category) and their bimodules, and show their dualities. In particular, a modular tensor category $\CC$ gives rise to a self enriched multi-fusion category $\FC=(\CC,\bar\CC)$ whose Drinfeld center was shown in \cite{KZ2} to be equivalent to $\CC$ itself. After these concrete mathematical results are established, we organize the enriched multi-fusion categories and their bimodules into a symmetric monoidal $(4,3)$-category with duals (see Theorem \ref{thm:dual}), and argue that the extended 3D TQFT $\CZ_\CC$ given by the fully dualizable object $\FC$ provides a way to extend the Reshetikhin-Turaev TQFT $\CZ^{RT}_\CC$ down to dimension zero. Actually, we will argue that the 1-2-3-dimensional theory of $\CZ_\CC$ is a combination of those of $\CZ^{RT}_\CC$ and $\CA_\CC$ and show that the double theory of $\CZ_\CC$ is equivalent to $\CZ^{TV}_\CC$.

\bigskip
\noindent {\bf Acknowledgement.} Work is supported by NSFC under Grant No. 11131008. The author would like to thank Liang Kong for many enlightening discussions.

\section{Unitary categories}

In this section, we recall some basic facts about unitary multi-fusion categories and unitary modular tensor categories, and set our notations.

%\subsection{Unitary multi-fusion categories} \label{sec:umfc}

\smallskip

A {\em $*$-category} $\CC$ is a $\C$-linear category equipped a $*$-operation on morphism, i.e. $*:\Hom_\CC(x,y) \to \Hom_\CC(y,x)$ is defined so that $(g \circ f)^* = f^* \circ g^*$, $(\lambda f)^* = \bar{\lambda} f^*$, $f^{**} = f$ for $f\in \Hom_\CC(x,y)$, $g\in \Hom_\CC(y,z)$, $\lambda \in \C$.
An isomorphism $f$ in $\CC$ is {\em unitary} if $f^*=f^{-1}$.
A {\em $*$-functor} $F: \CC\to \CD$ between two $*$-categories is a $\C$-linear functor such that $F(f^*)=F(f)^*$ for all morphisms $f$ in $\CC$.

An example of $*$-category is the category of finite-dimensional Hilbert spaces, denoted by $\bk$. A {\em unitary category} is a $*$-category $\CC$ which is equivalent via a $*$-functor to a finite direct sum of $\bk$.
We always assume a functor $F: \CC\to \CD$ between two unitary categories is an additive $*$-functor.
The Deligne tensor product $\CC\boxtimes\CD$ \cite{deligne} of two unitary categories $\CC$ and $\CD$ is automatically unitary.

\smallskip

A {\em unitary multi-fusion category} (UMFC) is a rigid monoidal category $\CC$ such that $\CC$ is equipped with the structure of a unitary category, the tensor product is an additive $*$-functor $\otimes: \CC \boxtimes \CC \to \CC$ and the natural isomorphisms $\one\otimes a \simeq a$,  $a\otimes \one \simeq a$, $(a\otimes b) \otimes c \simeq a\otimes (b\otimes c)$ are unitary.
In this case, the left dual and the right dual of an object $a\in\CC$ are canonically identified, which we denote as $a^*$.
We always assume a monoidal functor $F:\CC\to\CD$ between two UMFC's is such that $F$ is an additive $*$-functor and the natural isomorphisms $F(\one_\CC)\simeq\one_\CD$, $F(a\otimes b)\simeq F(a)\otimes F(b)$ are unitary.

A UMFC is called a {\em unitary fusion category} if the tensor unit $\one$ is a simple object. A UMFC is {\em indecomposable} (IUMFC for short) if it is neither zero nor a direct sum of two nonzero UMFC's.
Given a UMFC $\CC$, we use $\CC^\rev$ to denote the UMFC which has the same underlying category $\CC$ but equipped with the reversed tensor product $a\otimes^\rev b := b\otimes a$.
%Given a UMFC $\CC$, we regard the opposite category $\CC^\op$ as a UMFC equipped with the same tensor product $\otimes$.
%The functor $\delta: a\mapsto a^*$ gives a canonical monoidal equivalences between $\CC^\rev$ and $\CC^\op$.

\smallskip

We always assume a left module $\CM$ (see \cite{ostrik}) over a UMFC $\CC$ is such that $\CM$ is a unitary category, the left $\CC$-action is an additive $*$-functor $\odot:\CC\boxtimes\CM\to\CM$ and the natural isomorphisms $\one\odot x\simeq x$, $(a\otimes b)\odot x \simeq a\odot(b\odot x)$ are unitary; and assume a left $\CC$-module functor $F:\CM\to\CN$ is such that $F$ is an additive $*$-functor and the natural isomorphisms $F(a\odot x)\simeq a\odot F(x)$ are unitary; and make similar assumptions on right modules and bimodules over UMFC's

Let $\CC,\CD$ be UMFC's. Given left $\CC$-modules $\CM$ and $\CN$, we use $\Fun_\CC(\CM,\CN)$ to denote the category of left $\CC$-module functors. It is also a unitary category \cite{eno2005,ghr}. Moreover, $\Fun_\CC(\CM,\CM)$ is a UMFC \cite{eno2005,ghr}.
In the special case $\CC=\bk$, we simply denote $\Fun_\bk(\CM,\CN)$ as $\Fun(\CM,\CN)$.
Given $\CC$-$\CD$-bimodules (equivalently, left $\CC\boxtimes\CD^\rev$-modules) $\CM$ and $\CN$, we use $\Fun_{\CC|\CD}(\CM,\CN)$ to denote the category of $\CC$-$\CD$-bimodule functors.

Given a $\CC$-$\CD$-bimodule $\CM$, the opposite category $\CM^\op$ is automatically a $\CD$-$\CC$-bimodule with the induced action $b\odot x\odot a := a^* \odot x \odot b^*$ for $a\in\CC$, $x\in\CM^\op$, $b\in\CD$. Note that the functor $a\mapsto a^*$ defines an equivalence of $\CC$-$\CC$-bimodules $\CC\simeq\CC^\op$.

\smallskip

For a left module $\CM$ over a UMFC $\CC$, the {\em internal hom} $[x,-]_\CC: \CM \to \CC$ for $x\in \CM$ is defined to be the right adjoint functor of $-\odot x: \CC \to \CM$, i.e. $\Hom_\CM(a\odot x, y)\simeq \Hom_\CC(a, [x,y]_\CC)$ for $a\in\CC,y\in\CM$. Since unitary categories are semisimple, such an internal hom always exists.

\smallskip

Let $\CC$ be a UMFC and let $\CM$ be a right $\CC$-module, $\CN$ a left $\CC$-module. The {\em tensor product} $\CM\boxtimes_\CC\CN$ is the universal unitary category that is equipped with a functor $\boxtimes_\CC:\CM\boxtimes\CN \to \CM\boxtimes_\CC\CN$ intertwining the $\CC$-actions (see, for example, \cite{tam,eno2009,dss2}). In the special case $\CC=\bk$, $\CM\boxtimes_\bk\CN$ is just the Deligne tensor product $\CM\boxtimes \CN$.

It turns out that the tensor product $\CM\boxtimes_\CC\CN$ always exists and is equivalent to $\Fun_\CC(\CM^\op,\CN)$ \cite{eno2009}. More explicitly, there is a canonical equivalence for left $\CC$-modules $\CM,\CN$ \cite[Corollary 2.2.5(1)]{KZ1}:
\begin{equation}\label{eq:xRy}
  \CM^\op\boxtimes_\CC\CN \simeq \Fun_\CC(\CM,\CN) \quad\text{defined by}\quad
  x\boxtimes_\CC y\mapsto [-,x]_\CC^*\odot y.
\end{equation}
Moreover, there is a canonical equivalence for a right $\CC$-module $\CM$ and a left $\CC$-module $\CN$ \cite[Corollary 2.2.5(3)]{KZ1}:
\begin{equation}\label{eq:op}
  (\CM\boxtimes_\CC\CN)^\op \simeq \CN^\op\boxtimes_\CC\CM^\op \quad\text{defined by}\quad
  x\boxtimes_\CC y\mapsto y\boxtimes_\CC x.
\end{equation}

\smallskip

Let $\CC,\CD$ be UMFC's and let $\CM$ be a $\CC$-$\CD$-bimodule, $\CN$ a $\CD$-$\CC$-bimodule. We say that $\CM$ is {\em left dual} to $\CN$ and $\CN$ is {\em right dual} to $\CM$, if there exist a $\CC$-$\CC$-bimodule functor $u:\CD\to\CN\boxtimes_\CC\CM$ and a $\CD$-$\CD$-bimodule functor $v:\CM\boxtimes_\CD\CN\to\CC$ such that the composite bimodule functors
$$\CM \simeq \CM\boxtimes_\CD\CD \xrightarrow{\Id_\CM\boxtimes_\CD u} \CM\boxtimes_\CD\CN\boxtimes_\CC\CM \xrightarrow{v\boxtimes_\CC\Id_\CM} \CC\boxtimes_\CC\CM \simeq \CM,$$
$$\CN \simeq \CD\boxtimes_\CD\CN \xrightarrow{u\boxtimes_\CD\Id_\CN} \CN\boxtimes_\CC\CM\boxtimes_\CD\CN \xrightarrow{\Id_\CN\boxtimes_\CC v} \CN\boxtimes_\CC\CC \simeq \CN$$
are isomorphic to the identity functors.
In the special case where $u,v$ are equivalences, we say that the bimodules $\CM,\CN$ are {\em invertible} and say that the UMFC's $\CC,\CD$ are {\em Morita equivalent}.

It was shown in \cite{dss1} that the left dual and the right dual of a $\CC$-$\CD$-bimodule $\CM$ are given by $\Fun_{\CD^\rev}(\CM,\CD)$ and $\Fun_\CC(\CM,\CC)$, respectively. Given below is another form of this result which might be useful for computation (see \cite[Theorem 4.6]{AKZ}).

\begin{thm} \label{thm:dual-mod-cat}
Let $\CC,\CD$ be UMFC's and let $\CM$ be a $\CC$-$\CD$-bimodule. Then the $\CD$-$\CC$-bimodule $\CM^\op$ is right dual to $\CM$ with two duality maps $u$ and $v$ defined as follows:
\begin{align*}
& u: \CD \to \Fun_\CC(\CM,\CM)\simeq \CM^\op\boxtimes_\CC\CM, \quad\quad  d\mapsto  -\odot d,    \\
& v: \CM\boxtimes_\CD\CM^\op\to \CC, \quad\quad\quad\quad\quad\quad\quad\quad  x\boxtimes_\CD y \mapsto [x,y]_\CC^*.
\end{align*}
Since $(\CM^\op)^\op \simeq \CM$ as $\CC$-$\CD$-bimodules, $\CM^\op$ is also left dual to $\CM$.
\end{thm}

\smallskip

%\subsection{Unitary modular tensor categories}

A {\em unitary braided fusion category} is a unitary fusion category $\CC$ equipped with a braiding $c_{a,b}: a\otimes b\to b\otimes a$ such that the isomorphisms $c_{a,b}$ are unitary.
%We use $\bar\CC$ to denote the unitary fusion category $\CC^\rev$ equipped with the anti-braiding $\bar c_{a,b} := c^{-1}_{a,b}$. % : a\otimes^\rev b\to b\otimes^\rev a$.
We use $\bar\CC$ to denote the same unitary fusion category $\CC$ but equipped with the anti-braiding $\bar c_{a,b} := c^{-1}_{b,a}$.

The {\em centralizer} of a subcategory $\CE$ in a unitary braided fusion category $\CC$, denoted by $\CE'$, is the full subcategory of $\CC$ consisting of those objects $x$ such that $c_{y,x}\circ c_{x,y}=\Id_{x\otimes y}$ for all $y\in\CE$. The centralizer $\CC'$ of $\CC$ itself is referred to as the {\em M\"{u}ger center} of $\CC$.
A unitary braided fusion category has a canonical spherical structure \cite{kitaev}.
A {\em unitary modular tensor category} (UMTC) is a unitary braided fusion category, equipped with the canonical spherical structure, such that $\CC' \simeq \bk$.

%\smallskip

If $\CC$ is an IUMFC, the Drinfeld center $Z(\CC)$ is a UMTC \cite{Mu2}. Note that the UMTC $Z(\CC^\rev)$ is identical to $\overline{Z(\CC)}$. If $\CC$ is a UMTC, then the canonical embeddings $\CC,\bar\CC\hookrightarrow Z(\CC)$ induce a braided monoidal equivalence $\CC\boxtimes\bar\CC \simeq Z(\CC)$ \cite{Mu2}. If $\CE\hookrightarrow\CC$ is a fully faithful braided monoidal functor between UMTC's, the centralizer $\CE'$ of $\CE$ in $\CC$ is also a UMTC and we have $\CC\simeq\CE\boxtimes\CE'$ \cite[Theorem 4.2]{Mu3} (see also \cite{DGNO}).

\smallskip

%We recall the notions of a multi-fusion bimodule introduced in \cite{KZ1}. Similar notions in much more general settings were introduced earlier in \cite{lurie,francis,ginot,bbj2}.

Let $\CC,\CD$ be UMTC's. A {\em multi-fusion $\CC$-$\CD$-bimodule} is a UMFC $\CM$ equipped with a braided monoidal functor $\phi_\CM:\bar\CC\boxtimes \CD\to Z(\CM)$.
(The standard terminology should be a UMFC over $\bar\CC\boxtimes\CD$, however, this nonstandard one is sometimes more convenient.)
A multi-fusion $\CC$-$\CD$-bimodule $\CM$ is said to be {\em closed} if $\phi_\CM$ is an equivalence.
%If a closed multi-fusion $\CC$-$\CD$-bimodule exists, then the UMTC's $\CC$ and $\CD$ are said to be {\em Witt equivalent} (\cite{dmno}).
We say that two multi-fusion $\CC$-$\CD$-bimodules $\CM$ and $\CN$ are {\em equivalent} if there is a monoidal equivalence $\CM\simeq\CN$ such that the composition of $\phi_\CM$ with the induced braided monoidal equivalence $Z(\CM)\simeq Z(\CN)$ is isomorphic to $\phi_\CN$.

Let $\CC,\CD,\CE$ be UMTC's, and let $\CM$ be a multi-fusion $\CC$-$\CD$-bimodule, $\CN$ a multi-fusion $\CD$-$\CE$-bimodule. The category $\CM\boxtimes_\CD \CN$ has a natural structure of a UMFC with the tensor product $(a\boxtimes_\CD b) \otimes (c\boxtimes_\CD d) := (a\otimes c)\boxtimes_\CD(b\otimes d)$. Moreover, $(\CM\boxtimes_\CD\CN)^\rev$ can be identified with $\CN^\rev\boxtimes_\CD\CM^\rev$.
%Moreover, $\CM\boxtimes_\CD\CN$ satisfies the usual universal property. More precisely, if $\CK$ is another UMFC and there is a monoidal functor $F: \CM\boxtimes \CN \to \CK$ such that $F$ is also a balanced $\CD$-module functor, then there is a unique monoidal functor $\tilde{F}: \CM\boxtimes_\CD \CN \to \CK$ up to isomorphism such that $F\simeq\tilde{F} \circ \boxtimes_\CD$.

\begin{thm}[\cite{KZ1} Theorem 3.3.6] \label{thm:closed}
Let $\CC,\CD,\CE$ be UMTC's. If $\CM$ is a closed multi-fusion $\CC$-$\CD$-bimodule and $\CN$ is a closed multi-fusion $\CD$-$\CE$-bimodule, then $\CM\boxtimes_\CD \CN$ is a closed multi-fusion $\CC$-$\CE$-bimodule.
\end{thm}

We recall the main result in \cite{KZ1} that can be generalized to the unitary case automatically. Let $\mathbf{IUMFC}$ be the category of IUMFC's with morphisms given by the equivalence classes of nonzero bimodules. Let $\mathbf{UMTC}$ be the category of UMTC's with morphisms given by the equivalence classes of closed multi-fusion bimodules. The composition law of both categories is tensor product of bimodules.

\begin{thm}[\cite{KZ1} Theorem 3.3.7] \label{thm:kz}
There is a well-defined functor $Z: \mathbf{IUMFC} \to \mathbf{UMTC}$ given by $\CC \mapsto Z(\CC)$ on object and ${}_\CC \CM_\CD \mapsto \Fun_{\CC|\CD}(\CM,\CM)$ on morphism. Moreover, the functor $Z$ is fully faithful.
\end{thm}

\begin{rem}
The fully-faithfulness of $Z$ has essentially been given in \cite{eno2008,eno2009,dmno}.
\end{rem}

More explicitly, Theorem \ref{thm:kz} implies the following result. Let $\CC,\CD,\CE$ be IUMFC's. Let $\CM$ be a nonzero $\CC$-$\CD$-bimodule and $\CN$ a nonzero $\CD$-$\CE$-bimodule. The assignment $f\boxtimes_{Z(\CD)}g \mapsto f\boxtimes_\CD g$ defines an equivalence between two closed multi-fusion $Z(\CC)$-$Z(\CE)$-bimodules:
\begin{equation}\label{eq:MMNN}
  \Fun_{\CC|\CD}(\CM, \CM) \boxtimes_{Z(\CD)} \Fun_{\CD|\CE}(\CN,\CN)
  \simeq \Fun_{\CC|\CE}(\CM\boxtimes_\CD \CN, \CM\boxtimes_\CD \CN).
\end{equation}

\begin{cor} \label{cor:uib}
Let $\CC,\CD$ be IUMFC's. Given a braided monoidal equivalence $Z(\CC)\simeq Z(\CD)$, there is a unique invertible $\CC$-$\CD$-bimodule $\CM$ up to equivalence such that $\Fun_{\CC|\CD}(\CM,\CM) \simeq Z(\CC)$ as multi-fusion $Z(\CC)$-$Z(\CD)$-bimodules. Moreover, $\CM$ is the unique $\CC$-$\CD$-bimodule up to equivalence such that the canonical monoidal functor $\CC\boxtimes_{Z(\CC)}\CD^\rev \to \Fun(\CM,\CM)$, $c\boxtimes_{Z(\CC)}d \mapsto c\odot-\odot d$ is an equivalence
\end{cor}

\begin{proof}
The first claim is an immediate consequence of the fully-faithfulness of $Z$. Moreover, applying \eqref{eq:MMNN} we have
\begin{align*}
\nonumber
\CC\boxtimes_{Z(\CC)}\CD^\rev
& \simeq \Fun_{\bk|\CC}(\CC,\CC)\boxtimes_{Z(\CC)}\Fun_{\CC|\CD}(\CM,\CM)\boxtimes_{Z(\CD)}\Fun_{\CD|\bk}(\CD,\CD) \\
& \simeq \Fun_{\bk|\bk}(\CC\boxtimes_\CC\CM\boxtimes_\CD\CD,\CC\boxtimes_\CC\CM\boxtimes_\CD\CD) \\
&\simeq \Fun(\CM,\CM).
\end{align*}
If $\CN$ is another $\CC$-$\CD$-bimodule such that $\CC\boxtimes_{Z(\CC)}\CD^\rev \simeq \Fun(\CN,\CN)$, then the monoidal equivalence $\Fun(\CM,\CM)\simeq\Fun(\CN,\CN)$ induces an equivalence $\CM\simeq\CN$ which is automatically a bimodule equivalence.
%To show the equivalence \eqref{eqn:uib}, note that $Z({}_\bk\CC_\CC)\simeq\CC$ and $Z({}_\CD\CD_\bk)\simeq\CD^\rev$. Therefore, $\CC\boxtimes_{Z(\CC)}\CD^\rev \simeq Z({}_\bk\CC_\CC)\boxtimes_{Z(\CC)}Z(\CM)\boxtimes_{Z(\CD)}Z({}_\CD\CD_\bk) \simeq Z({}_\bk\CC\boxtimes_\CC\CM\boxtimes_\CD\CD_\bk) \simeq \Fun(\CM,\CM)$.
\end{proof}

\begin{cor} \label{cor:inv-bimod}
Let $\CC,\CD$ be UMTC's and let $\CM$ be a closed multi-fusion $\CC$-$\CD$-bimodule. We have the following assertions:
\begin{enumerate}
\item The canonical monoidal functor $\CM\boxtimes_{\bar\CC\boxtimes\CD}\CM^\rev \to \Fun(\CM,\CM)$, $x\boxtimes_{\bar\CC\boxtimes\CD}y \mapsto x\odot-\odot y$ is an equivalence.
\item The $\CM\boxtimes_\CD\CM^\rev$-$\CC$ bimodule $\CM$ is invertible and $\Fun_{\CM\boxtimes_\CD\CM^\rev|\CC}(\CM,\CM) \simeq \bar\CC\boxtimes\CC$ as multi-fusion $Z(\CM\boxtimes_\CD\CM^\rev)$-$Z(\CC)$-bimodules.
\item The $\CD$-$\CM^\rev\boxtimes_{\CC}\CM$ bimodule $\CM$ is invertible and $\Fun_{\CD|\CM^\rev\boxtimes_{\CC}\CM}(\CM,\CM) \simeq \bar\CD\boxtimes\CD$ as multi-fusion $Z(\CD)$-$Z(\CM^\rev\boxtimes_{\CC}\CM)$-bimodules.
\end{enumerate}
\end{cor}

\begin{proof}
$(1)$ is an easy consequence of Corollary \ref{cor:uib}. Note that $\CM\boxtimes_\CD\CM^\rev$ is a closed multi-fusion $\CC$-$\CC$-bimodule by Theorem \ref{thm:closed}. Moreover,
$(\CM\boxtimes_\CD\CM^\rev)\boxtimes_{\bar\CC\boxtimes\CC}\CC^\rev
\simeq \CM\boxtimes_{\bar\CC\boxtimes\CD}(\CM^\rev\boxtimes_\CC\CC^\rev)
\simeq \CM\boxtimes_{\bar\CC\boxtimes\CD}(\CC\boxtimes_\CC\CM)^\rev
\simeq \CM\boxtimes_{\bar\CC\boxtimes\CD}\CM^\rev
\simeq \Fun(\CM,\CM)$.
Applying Corollary \ref{cor:uib} again, we obtain $(2)$. $(3)$ is proved similarly.
\end{proof}

\begin{cor} \label{cor:inv-comp}
Let $\CC,\CD,\CE$ be UMTC's and let $\CM,\CN$ be closed multi-fusion $\CC$-$\CD$-bimodules, $\CM',\CN'$ be closed multi-fusion $\CD$-$\CE$-bimodules. Suppose $\CF$ is an invertible $\CM$-$\CN$-bimodule such that $\Fun_{\CM|\CN}(\CF,\CF)\simeq\bar\CC\boxtimes\CD$ as multi-fusion $Z(\CM)$-$Z(\CN)$-bimodules, and $\CG$ is an invertible $\CM'$-$\CN'$-bimodule such that $\Fun_{\CM'|\CN'}(\CG,\CG)\simeq\bar\CD\boxtimes\CE$ as multi-fusion $Z(\CM')$-$Z(\CN')$-bimodules. We have the following assertions:
\begin{enumerate}
\item The canonical monoidal functor $\CM\boxtimes_{\bar\CC}\CN^\rev \to \Fun_{\CD^\rev}(\CF,\CF)$, $x\boxtimes_{\bar\CC}y \mapsto x\odot-\odot y$ is an equivalence.
\item The canonical monoidal functor $\CM'\boxtimes_{\CE}\CN'^\rev \to \Fun_{\CD}(\CG,\CG)$, $x\boxtimes_{\CE}y \mapsto x\odot-\odot y$ is an equivalence.
\item The $\CM\boxtimes_\CD\CM'$-$\CN\boxtimes_\CD\CN'$-bimodule $\CF\boxtimes_\CD\CG$ is invertible and $\Fun_{\CM\boxtimes_\CD\CM'|\CN\boxtimes_\CD\CN'}(\CF\boxtimes_\CD\CG,\CF\boxtimes_\CD\CG) \simeq \bar\CC\boxtimes\CE$ as multi-fusion $Z(\CM\boxtimes_\CD\CM')$-$Z(\CN\boxtimes_\CD\CN')$-bimodules.
\end{enumerate}
\end{cor}

\begin{proof}
Using a similar augment as the proof of Corollary \ref{cor:inv-bimod}, we deduce that the $\CD^\rev$-$(\CM\boxtimes_{\bar\CC}\CN^\rev)^\rev$-bimodule $\CF$ and the $\CD$-$(\CM'\boxtimes_{\CE}\CN'^\rev)^\rev$-bimodule $\CG$ are invertible. Applying Theorem \ref{thm:dual-mod-cat}, we obtain $(1)$ and $(2)$.
Moreover,
\begin{equation*}
\begin{split}
& (\CM\boxtimes_\CD\CM') \boxtimes_{\bar\CC\boxtimes\CE} (\CN\boxtimes_\CD\CN')^\rev \\
\simeq ~& (\CM\boxtimes_{\bar\CC}\CN^\rev) \boxtimes_{\CD\boxtimes\bar\CD} (\CM'\boxtimes_{\CE}\CN'^\rev)\\
\simeq ~& \Fun_{\bk|\CD}(\CF,\CF) \boxtimes_{Z(\CD)} \Fun_{\CD|\bk}(\CG,\CG) \\
\simeq ~& \Fun(\CF\boxtimes_\CD\CG,\CF\boxtimes_\CD\CG)
\end{split}
\end{equation*}
where the last $\simeq$ is due to \eqref{eq:MMNN}.
Applying Corollary \ref{cor:uib}, we obtain $(3)$.
\end{proof}

\section{Enriched IUMFC's and bimodules}

In this section, we introduced the notions of enriched IUMFC's, bimodules and bimodule functors, etc. For reader's convenience, we draw several figures to illustrate these concepts.

\begin{defn}
An {\em enriched IUMFC} is a pair $\FC=(\CC_1,\CC_2)$ where $\CC_2$ is a UMTC and $\CC_1$ is an IUMFC equipped with a fully faithful braided monoidal functor $\CC_2\hookrightarrow Z(\CC_1)$. See Figure \ref{fig:iumfc}(a).
We use $\FC^\rev$ to denote the enriched IUMFC $(\CC_1^\rev,\bar\CC_2)$. See Figure \ref{fig:iumfc}(b).
By abusing notation, we use $\bk$ to denote the enriched IUMFC $(\bk,\bk)$.
%By abusing notation, we use $\CC$ to denote the enriched IUMFC $(\CC,\bk)$ for an IUMFC $\CC$.
\end{defn}

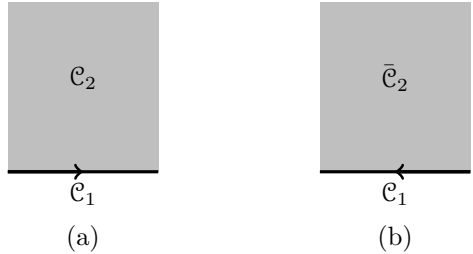
\begin{figure}[h]
\begin{tabular}{c@{\quad\quad\quad\quad\quad\quad}c}
\begin{tikzpicture}[scale=.25,line width=1.2pt]
\fill[lightgray] (0,0) rectangle (8,9);
\draw (4,5) node{$\CC_2$};
\draw (0,0) -- (8,0);
\draw (0,0) [->] -- (4,0) node[below]{$\CC_1$};
\end{tikzpicture}
&
\begin{tikzpicture}[scale=.25,line width=1.2pt]
\fill[lightgray] (0,0) rectangle (8,9);
\draw (4,5) node{$\bar\CC_2$};
\draw (0,0) -- (8,0);
\draw (8,0) [->] -- (4,0) node[below]{$\CC_1$};
\end{tikzpicture}
\\ (a) & (b)
\end{tabular}
\caption{(a) depicts an enriched IUMFC $\FC=(\CC_1,\CC_2)$ where $\CC_1$ admits a central action by $\CC_2$. (b) depicts its reverse $\FC^\rev=(\CC_1^\rev,\bar\CC_2)$. Keep in mind that the categorical constructions $\CC\mapsto\bar\CC$, $\CC\mapsto\CC^\rev$, $\CC\mapsto\CC^\op$ correspond to changing the orientations of 2,1,0-cells in figures.}
\label{fig:iumfc}
\end{figure}

\begin{rem} \label{rem:enrich}
It was shown in \cite{KZ2} that an enriched IUMFC $\FC=(\CC_1,\CC_2)$ is naturally associated with a monoidal category $\FC^\sharp$ enriched in $\bar\CC_2$ in the sense of \cite{MP} such that $\FC^\sharp$ takes $\CC_1$ as underlying category and $\Hom_{\FC^\sharp}(x,y) = [x,y]_{\bar\CC_2}$. Moreover, the Drinfeld center $Z(\FC^\sharp)$ was shown to be equivalent to the centralizer of $\CC_2$ in $Z(\CC_1)$. In particular, $Z(\CC_1)\simeq\CC_2\boxtimes Z(\FC^\sharp)$. In the special case $\FC=(\CC,\bar\CC)$ where $\CC$ is a UMTC, note that $Z(\FC^\sharp)\simeq\CC$.
%In \cite{KZ2}, we introduced the notion of Drinfeld center of an enriched monoidal category (see \cite{MP}) and showed that the Drinfeld center of a self enriched modular tensor category $\CC$ is equivalent to $\CC$ itself. In an expanded version of \cite{KZ2}, we will generalize this result to show that an enriched IUMFC $\FC=(\CC_1,\CC_2)$ is associated with a monoidal category $\FC^\sharp$ enriched in $\bar\CC_2$ such that $\FC^\sharp$ takes $\CC_1$ as underlying category and $\Hom_{\FC^\sharp}(x,y) = [x,y]_{\bar\CC_2}$. Moreover, the Drinfeld center of $\FC^\sharp$ is equivalent to the centralizer of $\CC_2$ in $Z(\CC_1)$. At present, we would like to take this as a definition. That is, we define the Drinfeld center $Z(\FC^\sharp)$ to be the centralizer of $\CC_2$ in $Z(\CC_1)$. In particular, $Z(\CC_1)\simeq\CC_2\boxtimes Z(\FC^\sharp)$.
\end{rem}

\begin{defn}
Let $\FC$ and $\FD$ be enriched IUMFC's. The {\em Deligne tensor product} $\FC\boxtimes\FD$ is defined to be the enriched IUMFC $(\CC_1\boxtimes\CD_1,\CC_2\boxtimes\CD_2)$.
\end{defn}

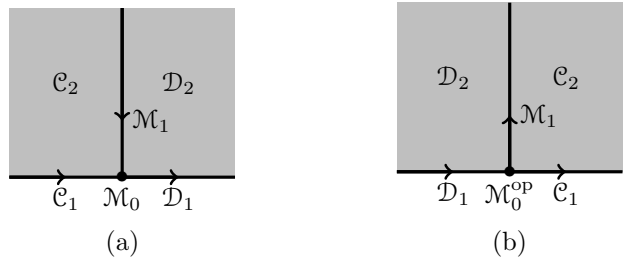
\begin{figure}[h]
\begin{tabular}{c@{\quad\quad\quad\quad\quad\quad}c}
\begin{tikzpicture}[scale=.25,line width=1.2pt]
\fill[lightgray] (0,0) rectangle (12,9);
\draw (3,5) node{$\CC_2$};
\draw (9,5) node{$\CD_2$};
\draw (6,0) -- (6,9);
\draw (6,9) [->] -- (6,3) node[right]{$\CM_1$};
\draw (0,0) -- (12,0);
\draw (0,0) [->] -- (3,0) node[below]{$\CC_1$};
\draw (6,0) node{$\bullet$} node[below]{$\CM_0$};
\draw (6,0) [->] -- (9,0) node[below]{$\CD_1$};
\end{tikzpicture}
&
\begin{tikzpicture}[scale=.25,line width=1.2pt]
\fill[lightgray] (0,0) rectangle (12,9);
\draw (3,5) node{$\CD_2$};
\draw (9,5) node{$\CC_2$};
\draw (6,0) -- (6,9);
\draw (6,0) [->] -- (6,3) node[right]{$\CM_1$};
\draw (0,0) -- (12,0);
\draw (0,0) [->] -- (3,0) node[below]{$\CD_1$};
\draw (6,0) node{$\bullet$} node[below]{$\CM_0^\op$};
\draw (6,0) [->] -- (9,0) node[below]{$\CC_1$};
\end{tikzpicture}
\\ (a) & (b)
\end{tabular}
\caption{(a) depicts a $\FC$-$\FD$-bimodule $\FM=(\CM_0,\CM_1)$ where $\CM_1$ admits actions by $\CC_2,\CD_2$, and $\CM_0$ admits actions by $\CC_2,\CM_1,\CD_2,\CC_1,\CD_1$. (b) depicts the opposite $\FD$-$\FC$-bimodule $\FM^\op=(\CM_0^\op,\CM_1^\rev)$.}
\label{fig:bimod}
\end{figure}

\begin{defn}
Let $\FC$ and $\FD$ be enriched IUMFC's. A {\em $\FC$-$\FD$-bimodule} is a pair $\FM=(\CM_0,\CM_1)$ where $\CM_1$ is a closed multi-fusion $\CC_2$-$\CD_2$-bimodule and $\CM_0$ is a left $\CC_1\boxtimes_{\CC_2}\CM_1\boxtimes_{\CD_2}\CD_1^\rev$-module. See Figure \ref{fig:bimod}(a).
We use $\FM^\op$ to denote the $\FD$-$\FC$-bimodule $(\CM_0^\op,\CM_1^\rev)$. See Figure \ref{fig:bimod}(b).
\end{defn}

%A bimodule $\FM$ defines a category $\FM^\sharp$ enriched in $\CM_1$ by regarding $\CM_0$ as a left $\CM_1$-module.

\begin{rem} \label{rem:dual-obj}
A $\FC$-$\FD$-bimodule is automatically a $\FD^\rev$-$\FC^\rev$-bimodule, a $\FC\boxtimes\FD^\rev$-$\bk$-bimodule as well as an $\bk$-$\FC^\rev\boxtimes\FD$-bimodule.
\end{rem}

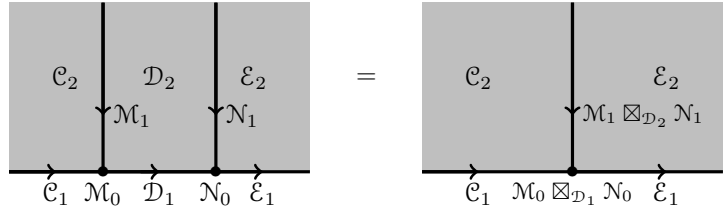
\begin{figure}[h]
\begin{tikzpicture}[scale=.25,line width=1.2pt]
\fill[lightgray] (0,0) rectangle (16,9);
\draw (3,5) node{$\CC_2$};
\draw (8,5) node{$\CD_2$};
\draw (13,5) node{$\CE_2$};
\draw (5,0) -- (5,9);
\draw (5,9) [->] -- (5,3) node[right]{$\CM_1$};
\draw (11,0) -- (11,9);
\draw (11,9) [->] -- (11,3) node[right]{$\CN_1$};
\draw (0,0) -- (16,0);
\draw (0,0) [->] -- (2.5,0) node[below]{$\CC_1$};
\draw (5,0) node{$\bullet$} node[below]{$\CM_0$};
\draw (0,0) [->] -- (8,0) node[below]{$\CD_1$};
\draw (11,0) node{$\bullet$} node[below]{$\CN_0$};
\draw (7,0) [->] -- (13.5,0) node[below]{$\CE_1$};
\draw (19,5) node{$=$};
\end{tikzpicture}
\quad
\begin{tikzpicture}[scale=.25,line width=1.2pt]
\fill[lightgray] (0,0) rectangle (16,9);
\draw (3,5) node{$\CC_2$};
\draw (13,5) node{$\CE_2$};
\draw (8,0) -- (8,9);
\draw (8,9) [->] -- (8,3) node[right]{\small $\CM_1\boxtimes_{\CD_2}\CN_1$};
\draw (0,0) -- (16,0);
\draw (0,0) [->] -- (3,0) node[below]{$\CC_1$};
\draw (8,0) node{$\bullet$} node[below]{\small $\CM_0\boxtimes_{\CD_1}\CN_0$};
\draw (8,0) [->] -- (13,0) node[below]{$\CE_1$};
\end{tikzpicture}
\caption{The tensor product $\FM\boxtimes_\FD\FN = (\CM_0\boxtimes_{\CD_1}\CN_0,\CM_1\boxtimes_{\CD_2}\CN_1)$ is depicted by either of these two equivalent figures. The right one is obtained from the left one by contracting the region labelled by $\CD_2$ and $\CD_1$. Such an equivalence under contraction will be implicitly used in the following figures.}
\label{fig:bimod-tensor}
\end{figure}

\begin{defn}
Let $\FC$, $\FD$, $\FE$ be enriched IUMFC's and let $\FM$ be a $\FC$-$\FD$-bimodule, $\FN$ be a $\FD$-$\FE$-bimodule. The {\em tensor product} $\FM\boxtimes_\FD\FN$ is defined to be the $\FC$-$\FE$-bimodule $(\CM_0\boxtimes_{\CD_1}\CN_0,\CM_1\boxtimes_{\CD_2}\CN_1)$. See Figure \ref{fig:bimod-tensor}. This is well defined because $\CM_1\boxtimes_{\CD_2}\CN_1$ is a closed multi-fusion $\CC_2$-$\CE_2$-bimodule by Theorem \ref{thm:closed}.
\end{defn}

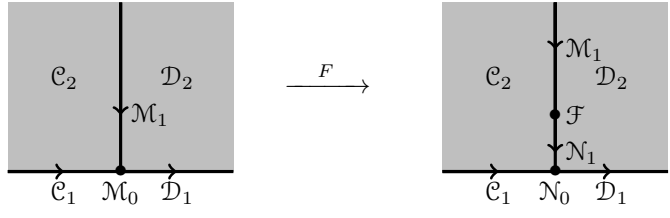
\begin{figure}[h]
\begin{tikzpicture}[scale=.25,line width=1.2pt]
\fill[lightgray] (0,0) rectangle (12,9);
\draw (3,5) node{$\CC_2$};
\draw (9,5) node{$\CD_2$};
\draw (6,0) -- (6,9);
\draw (6,9) [->] -- (6,3) node[right]{$\CM_1$};
\draw (0,0) -- (12,0);
\draw (0,0) [->] -- (3,0) node[below]{$\CC_1$};
\draw (6,0) node{$\bullet$} node[below]{$\CM_0$};
\draw (6,0) [->] -- (9,0) node[below]{$\CD_1$};
\draw (17,5) node{$\xrightarrow{\quad F\quad}$};
\end{tikzpicture}
\quad\quad
\begin{tikzpicture}[scale=.25,line width=1.2pt]
\fill[lightgray] (0,0) rectangle (12,9);
\draw (3,5) node{$\CC_2$};
\draw (9,5) node{$\CD_2$};
\draw (6,0) -- (6,9);
\draw (6,9) [->] -- (6,6.5) node[right]{$\CM_1$};
\draw (6,3) node{$\bullet$} node[right]{$\CF$};
\draw (6,9) [->] -- (6,1) node[right]{$\CN_1$};
\draw (0,0) -- (12,0);
\draw (0,0) [->] -- (3,0) node[below]{$\CC_1$};
\draw (6,0) node{$\bullet$} node[below]{$\CN_0$};
\draw (6,0) [->] -- (9,0) node[below]{$\CD_1$};
\end{tikzpicture}
\caption{This figure depicts a $\FC$-$\FD$-bimodule functor $\FF=(F,\CF)$.}
\label{fig:fun}
\end{figure}

\begin{defn}
Let $\FC$, $\FD$ be enriched IUMFC's and $\FM$, $\FN$ be $\FC$-$\FD$-bimodules.
A {\em bimodule functor} $\FF: \FM \to \FN$ is a pair $\FF=(F,\CF)$ where $\CF$ is an invertible $\CM_1$-$\CN_1$-bimodule such that $\Fun_{\CM_1|\CN_1}(\CF,\CF) \simeq \bar\CC_2\boxtimes\CD_2$ as multi-fusion $Z(\CM_1)$-$Z(\CN_1)$-bimodules and $F$ is a left $\CC_1\boxtimes_{\CC_2}\CM_1\boxtimes_{\CD_2}\CD_1^\rev$-module functor $F: \CM_0 \to \CF\boxtimes_{\CN_1}\CN_0$. See Figure \ref{fig:fun}.
\end{defn}

\begin{rem} \label{rem:bimod-func}
According to Theorem \ref{thm:kz}, an invertible $\CM_1$-$\CN_1$-bimodule $\CF$ in the definition always exists and is unique up to equivalence.
%In the special case where $\CC_2=\CM_1=\CN_1=\CD_2$, we have a trivial $\CM_1$-$\CN_1$-bimodule $\CN_1$. Under the choice $\CF=\CN_1$, a bimodule functor $\FF: \FM \to \FN$ is simply a left $\CC_1\boxtimes_{\CC_2}\CD_1^\rev$-module functor $F: \CM_0 \to \CN_0$.
\end{rem}

\begin{rem}
Note that an IUMFC $\CC$ can be regarded as an enriched IUMFC $(\CC,\bk)$ and a $\CC$-$\CD$-bimodule $\CM$ over IUMFC's $\CC$ and $\CD$ defines a $(\CC,\bk)$-$(\CD,\bk)$-bimodule $(\CM,\bk)$. Moreover, if $\CN$ is another $\CC$-$\CD$-bimodule, then a bimodule functor $F:\CM\to\CN$ defines a bimodule functor $(F,\bk):(\CM,\bk)\to(\CN,\bk)$. So, the definitions introduced above indeed generalize those notions in unenriched case.
\end{rem}

\begin{defn}
Let $\FM$, $\FN$, $\FL$ be $\FC$-$\FD$-bimodules and let $\FF: \FM \to \FN$ and $\FG: \FN \to \FL$ be bimodule functors. The {\em composition} $\FG\circ\FF: \FM \to \FL$ is defined to be the bimodule functor given by the composite left module functor
$$\CM_0 \xrightarrow{F} \CF\boxtimes_{\CN_1}\CN_0 \xrightarrow{\Id_\CF\boxtimes_{\CN_1}G} (\CF\boxtimes_{\CN_1}\CG) \boxtimes_{\CL_1} \CL_0.
$$
This is well defined due to Theorem \ref{thm:kz}, i.e. $\CF\boxtimes_{\CN_1}\CG$ satisfies the condition required by a bimodule functor.
\end{defn}

\begin{defn}
Let $\FM$, $\FN$ be $\FC$-$\FD$-bimodules, $\FM',\FN'$ be $\FD$-$\FE$-bimodules, and let $\FF: \FM \to \FN$ and $\FG: \FM' \to \FN'$ be bimodule functors. The {\em tensor product} $\FF\boxtimes_\FD\FG: \FM\boxtimes_\FD\FM' \to \FN\boxtimes_\FD\FN'$ is defined to be the bimodule functor $(F\boxtimes_{\CD_1}G,\CF\boxtimes_{\CD_2}\CG)$.
%given by the left module functor $$\CM_0\boxtimes_{\CD_1}\CM'_0 \xrightarrow{F\boxtimes_{\CD_1}G} (\CF\boxtimes_{\CN_1}\CN_0)\boxtimes_{\CD_1}(\CG\boxtimes_{\CN'_1}\CN'_0) \simeq (\CF\boxtimes_{\CD_2}\CG) \boxtimes_{\CN_1\boxtimes_{\CD_2}\CN'_1} (\CN_0\boxtimes_{\CD_1}\CN'_0).$$
This is well defined in view of Corollary \ref{cor:inv-comp}(3).
\end{defn}

\begin{defn}
Let $\FM$, $\FN$ be $\FC$-$\FD$-bimodules and let $\FF,\FG:\FM\to\FN$ be bimodule functors. A {\em natural transformation} $\xi:\FF\to\FG$ consists of an $\CM_1$-$\CN_1$-bimodule equivalence $\xi':\CF\to\CG$ and a natural transformation of left module functors $\xi: (\xi'\boxtimes_{\CN_1}\Id_{\CN_0})\circ F \to G$. A {\em modification} between two natural transformations $\xi,\zeta:\FF\to\FG$ is an isomorphism $\gamma:\xi'\to\zeta'$ such that $\xi = \zeta\circ((\gamma\boxtimes_{\CN_1}\Id_{\Id_{\CN_0}})\circ\Id_F)$.
\end{defn}

\begin{prop} \label{prop:inv-equiv0}
Let $\xi:\FF\to\FG$ be a natural transformation between two bimodule functors $\FF,\FG:\FM\to\FN$. The following conditions are equivalent:
\begin{enumerate}
\item $\xi$ is an equivalence, i.e. there is a natural transformation $\zeta:\FG\to\FF$ such that $\zeta\circ\xi\simeq\Id_\FF$ and $\xi\circ\zeta\simeq\Id_\FG$.
\item The natural transformation of left module functors $\xi: (\xi'\boxtimes_{\CN_1}\Id_{\CN_0})\circ F \to G$ is an isomorphism.
\end{enumerate}
\end{prop}

\begin{proof}
$(1)\Rightarrow(2)$. The natural transformation of left module functors $G \simeq (\xi'\boxtimes_{\CN_1}\Id_{\CN_0})\circ(\zeta'\boxtimes_{\CN_1}\Id_{\CN_0})\circ G \xrightarrow{\Id\circ\zeta} (\xi'\boxtimes_{\CN_1}\Id_{\CN_0})\circ F$ is inverse to $\xi$.

$(2)\Rightarrow(1)$. The desired natural transformation $\zeta:\FG\to\FF$ is given by $(\xi'^{-1}\boxtimes_{\CN_1}\Id_{\CN_0})\circ G \xrightarrow{\Id\circ\xi^{-1}} (\xi'^{-1}\boxtimes_{\CN_1}\Id_{\CN_0})\circ(\xi'\boxtimes_{\CN_1}\Id_{\CN_0})\circ F = F$.
\end{proof}

\begin{prop} \label{prop:inv-equiv}
Let $\FF:\FM\to\FN$ be a bimodule functor between $\FC$-$\FD$-bimodules. The following conditions are equivalent:
\begin{enumerate}
\item $\FF$ is a bimodule equivalence, i.e. there is a bimodule functor $\FG:\FN\to\FM$ such that $\FG\circ\FF\simeq\Id_\FM$ and $\FF\circ\FG\simeq\Id_\FN$.
\item The left module functor $F: \CM_0 \to \CF\boxtimes_{\CN_1}\CN_0$ is an equivalence.
\end{enumerate}
\end{prop}

\begin{proof}
$(1)\Rightarrow(2)$. The left module functor $\CF\boxtimes_{\CN_1}\CN_0 \xrightarrow{\Id_\CF\boxtimes_{\CN_1}G} \CF\boxtimes_{\CN_1}\CG\boxtimes_{\CM_1}\CM_0 \simeq \CM_0$ is inverse to $F$.

$(2)\Rightarrow(1)$. The desired bimodule functor $\FG:\FN\to\FM$ is given by the left module functor $\CN_0 \simeq \CF^\op\boxtimes_{\CM_1}\CF\boxtimes_{\CN_1}\CN_0 \xrightarrow{\Id_{\CF^\op}\boxtimes_{\CM_1}F^{-1}} \CF^\op\boxtimes_{\CM_1}\CM_0$.
\end{proof}

\begin{exam} \label{exam:trivial}
Let $\CC$ be a unitary category. We have an $\bk$-$\bk$-bimodule equivalence $(\CC,\Fun(\CC,\CC)) \simeq \bk$ given by the left $\Fun(\CC,\CC)$-module equivalence $\CC\simeq\CC\boxtimes_\bk\bk$.
%In fact, we have a bimodule functor $F: \bk \to (\CC,\Fun(\CC,\CC))$ given by the left $\bk$-module equivalence $\bk \simeq \CC\boxtimes_{\Fun(\CC,\CC)}\CC$ and a bimodule functor $G: (\CC,\Fun(\CC,\CC)) \to \bk$ given by the left $\Fun(\CC,\CC)$-module equivalence $\CC\simeq\CC\boxtimes_\bk\bk$. It is clear that $F$ and $G$ are inverse to each other.
\end{exam}

\begin{figure}[h]
\begin{tikzpicture}[scale=.25,line width=1.2pt]
\fill[lightgray] (0,0) rectangle (12,9);
\draw (3,5) node{$\CC_2$};
\draw (9,6) node{$\CD_2$};
\draw (6,9) -- (6,3) -- (12,3);
\draw (6,9) [->] -- (6,6.5) node[left]{$\CM_1$};
\draw (6,3) node{$\bullet$} node[left]{$\CM_1$};
\draw (6,3) [->] -- (9,3) node[above]{$\CM_1$};
\draw[dashed] (6,3) -- (6,0);
%\draw[dashed] [->] (6,3) -- (6,1.5) node[right]{$\CC_2$};
\draw (9,1.5) node{$\CC_2$};
\draw (0,0) -- (12,0);
\draw (0,0) [->] -- (3,0) node[below]{$\CC_1$};
\draw (6,0) node{$\bullet$} node[below]{$\CC_1$};
\draw (6,0) [->] -- (9,0) node[below]{$\CC_1$};
\end{tikzpicture}
\quad\quad\quad\quad\quad\quad
\begin{tikzpicture}[scale=.25,line width=1.2pt]
\fill[lightgray] (0,0) rectangle (12,9);
\draw (9,5) node{$\CC_2$};
\draw (3,6) node{$\CD_2$};
\draw (6,9) -- (6,3) -- (0,3);
\draw (6,3) [->] -- (6,6.5) node[right]{$\CM_1$};
\draw (6,3) node{$\bullet$} node[right]{$\CM_1^\op$};
\draw (0,3) [->] -- (3,3) node[above]{$\CM_1$};
\draw[dashed] (6,3) -- (6,0);
%\draw[dashed] [->] (6,3) -- (6,1.5) node[right]{$\CC_2$};
\draw (3,1.5) node{$\CC_2$};
\draw (0,0) -- (12,0);
\draw (0,0) [->] -- (3,0) node[below]{$\CC_1$};
\draw (6,0) node{$\bullet$} node[below]{$\CC_1^\op$};
\draw (6,0) [->] -- (9,0) node[below]{$\CC_1$};
\end{tikzpicture}
\caption{A $\FC$-$\FD$-bimodule $\FM=(\CC_1\boxtimes_{\CC_2}\CM_1,\CM_1)$ is depicted on the left where $\FC=(\CC_1,\CC_2)$ and $\FD=(\CC_1\boxtimes_{\CC_2}\CM_1,\CD_2)$. The $\FD$-$\FC$-bimodule $\FM^\op \simeq (\CM_1^\op\boxtimes_{\CC_2}\CC_1^\op,\CM_1^\rev)$ is depicted on the right.}
\label{fig:bimod-equiv}
\end{figure}
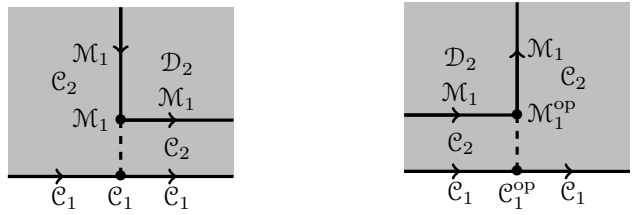

\begin{figure}[h]
\begin{tikzpicture}[scale=.3,line width=1.2pt]
\fill[lightgray] (0,0) rectangle (12,9);
\draw (1.5,5) node{$\CC_2$};
\draw (6,6) node{$\CD_2$};
\draw (10.5,5) node{$\CC_2$};
\draw (6,1.5) node{$\CC_2$};
\draw (4,9) -- (4,3) -- (8,3) -- (8,9);
\draw[dashed] (4,9) [->] -- (4,6.5) node[left]{$\CM_1$};
\draw[dashed] (4,3) [->] -- (6,3) node[above]{$\CM_1$};
\draw[dashed] (8,3) [->] -- (8,6.5) node[right]{$\CM_1$};
\draw (4,3) node{$\bullet$} node[left]{$\CM_1$};
\draw (8,3) node{$\bullet$} node[right]{$\CM_1^\op$};
\draw[dashed] (4,3) -- (4,0);
\draw[dashed] (8,3) -- (8,0);
\draw (0,0) -- (12,0);
\draw (0,0) [->] -- (1.5,0) node[below]{$\CC_1$};
\draw (4,0) node{$\bullet$} node[below]{$\CC_1$};
\draw (0,0) [->] -- (6,0) node[below]{$\CC_1$};
\draw (8,0) node{$\bullet$} node[below]{$\;\CC_1^\op$};
\draw (6,0) [->] -- (10.5,0) node[below]{$\CC_1$};
\draw (17,5) node{$\xrightarrow[\simeq]{\quad\alpha\quad}$};
\end{tikzpicture}
\quad\quad
\begin{tikzpicture}[scale=.3,line width=1.2pt]
\fill[lightgray] (0,0) rectangle (12,9);
\draw (2.5,3) node{$\CC_2$};
\draw (9.5,3) node{$\CC_2$};
\draw (6,7.5) node{$\CD_2$};
\draw[dashed] (6,5) -- (6,0);
\draw[dashed] (6,5) [->] -- (6,2) node[right]{$\CC_2$};
\draw (4,9) -- (6,5);
\draw (4,9) [->] -- (5,7) node[left]{$\CM_1$};
\draw (6,5) -- (8,9);
\draw (6,5) [->] -- (7,7) node[right]{$\CM_1$};
\draw (6,5) node{$\bullet$} node[right]{$\CM_1$};
\draw (0,0) -- (12,0);
\draw (0,0) [->] -- (3,0) node[below]{$\CC_1$};
\draw (6,0) node{$\bullet$} node[below]{$\CC_1$};
\draw (6,0) [->] -- (9,0) node[below]{$\CC_1$};
\end{tikzpicture}
\\~\\
\begin{tikzpicture}[scale=.3,line width=1.2pt]
\fill[lightgray] (0,0) rectangle (12,9);
\draw (3,6) node{$\CD_2$};
\draw (3,1.5) node{$\CC_2$};
\draw (9,6) node{$\CD_2$};
\draw (9,1.5) node{$\CC_2$};
\draw[dashed] (6,9) -- (6,0);
\draw[dashed] (6,9) [->] -- (6,7) node[right]{$\CD_2$};
\draw (0,3) -- (12,3);
\draw (0,3) [->] -- (3,3) node[above]{$\CM_1$};
\draw (6,3) node{$\bullet$} node[above]{$\CM_1$};
\draw (6,3) [->] -- (9,3) node[above]{$\CM_1$};
\draw (0,0) -- (12,0);
\draw (0,0) [->] -- (3,0) node[below]{$\CC_1$};
\draw (6,0) node{$\bullet$} node[below]{$\CC_1$};
\draw (6,0) [->] -- (9,0) node[below]{$\CC_1$};
\draw (17,5) node{$\xrightarrow[\simeq]{\quad\beta\quad}$};
\end{tikzpicture}
\quad\quad
\begin{tikzpicture}[scale=.3,line width=1.2pt]
\fill[lightgray] (0,0) rectangle (12,9);
\draw (1.5,7) node{$\CD_2$};
\draw (10.5,7) node{$\CD_2$};
\draw (6,1.5) node{$\CC_2$};
\draw[dashed] (6,9) -- (6,6);
\draw[dashed] (6,9) [->] -- (6,7.5) node[right]{$\CD_2$};
\draw (0,3) -- (4,3) -- (6,6) -- (8,3) -- (12,3);
\draw (0,3) [->] -- (1.5,3) node[above]{$\CM_1$};
\draw (4,3) [->] -- (5,4.5) node[left]{$\CM_1$};
\draw (6,6) [->] -- (7,4.5) node[right]{$\CM_1$};
\draw (9,3) [->] -- (10.5,3) node[above]{$\CM_1$};
\draw (4,3) node{$\bullet$} node[below left]{$\CM_1^\op$};
\draw (8,3) node{$\bullet$} node[below right]{$\CM_1$};
\draw[dashed] (4,0) -- (4,3);
\draw[dashed] (8,3) -- (8,0);
\draw (6,6) node{$\bullet$} node[right]{$\CM_1$};
\draw (0,0) -- (12,0);
\draw (0,0) [->] -- (1.5,0) node[below]{$\CC_1$};
\draw (4,0) node{$\bullet$} node[below]{$\CC_1^\op$};
\draw (0,0) [->] -- (6.3,0) node[below]{$\CC_1$};
\draw (8,0) node{$\bullet$} node[below]{$\CC_1$};
\draw (6,0) [->] -- (10.5,0) node[below]{$\CC_1$};
\end{tikzpicture}
\caption{These figures depict two bimodule equivalences $(\alpha,\CM_1): \FM\boxtimes_\FD\FM^\op \to \FC$ and $(\beta,\CM_1): \FD \to \FM^\op\boxtimes_\FC\FM$.}
\label{fig:bimod-equiv2}
\end{figure}
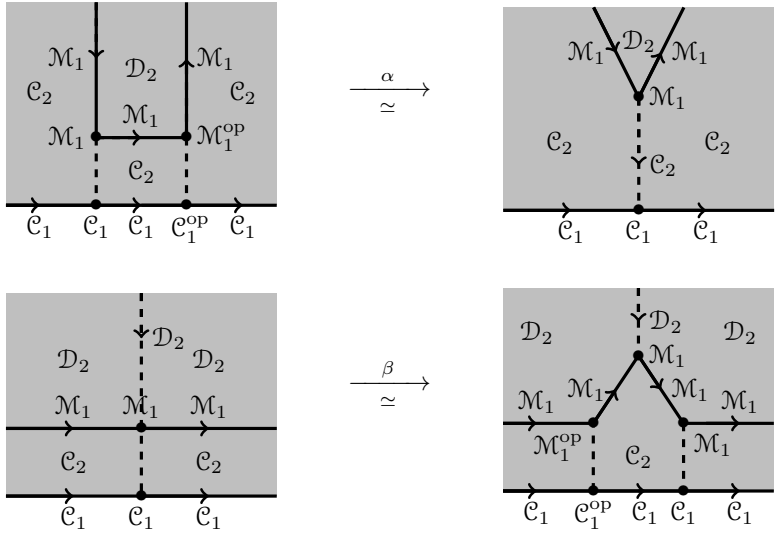

\begin{prop} \label{prop:inv-mor}
Let $\FM$ be a $\FC$-$\FD$-bimodule. Suppose $\CM_0 = \CD_1 = \CC_1\boxtimes_{\CC_2}\CM_1$ (see Figure \ref{fig:bimod-equiv}). Then the $\FD$-$\FC$-bimodule $\FM^\op$ is inverse to $\FM$. That is, $\FM\boxtimes_\FD\FM^\op\simeq\FC$ as $\FC$-$\FC$-bimodules and $\FM^\op\boxtimes_\FC\FM\simeq\FD$ as $\FD$-$\FD$-bimodules.
\end{prop}

\begin{proof}
We have $\FM\boxtimes_\FD\FM^\op = (\CM_0\boxtimes_{\CD_1}\CM_0^\op,\CM_1\boxtimes_{\CD_2}\CM_1^\rev) \simeq (\CD_1^\op,\CM_1\boxtimes_{\CD_2}\CM_1^\rev)$.
According to Corollary \ref{cor:inv-bimod}(2) and Proposition \ref{prop:inv-equiv}, the invertible $\CM_1\boxtimes_{\CD_2}\CM_1^\rev$-$\CC_2$ bimodule $\CM_1$ and the left $\CD_1\boxtimes_{\CD_2}\CD_1^\rev$-module equivalence $\alpha: \CD_1^\op = (\CC_1\boxtimes_{\CC_2}\CM_1)^\op \simeq \CM_1^\op\boxtimes_{\CC_2}\CC_1^\op \simeq \CM_1\boxtimes_{\CC_2}\CC_1$ defines a $\FC$-$\FC$-bimodule equivalence $\FM\boxtimes_\FD\FM^\op \simeq \FC$. See the top figure in Figure \ref{fig:bimod-equiv2}.

We have $\FM^\op\boxtimes_\FC\FM
= (\CM_0^\op\boxtimes_{\CC_1}\CM_0,\CM_1^\rev\boxtimes_{\CC_2}\CM_1)
\simeq (\CM_1^\op\boxtimes_{\CC_2}\CC_1^\op\boxtimes_{\CC_2}\CM_1,\CM_1^\rev\boxtimes_{\CC_2}\CM_1)$.
According to Corollary \ref{cor:inv-bimod}(3) and Proposition \ref{prop:inv-equiv}, the invertible $\CD_2$-$\CM_1^\rev\boxtimes_{\CC_2}\CM_1$-bimodule $\CM_1$ and the left $\CD_1\boxtimes_{\CD_2}\CD_1^\rev$-module equivalence
$\beta: \CD_1 \simeq \CD_1^\op
\simeq \CM_1^\op\boxtimes_{\CC_2}\CC_1^\op
\simeq \CM_1\boxtimes_{\CM_1^\rev\boxtimes_{\CC_2}\CM_1}(\CM_1^\op\boxtimes_{\CC_2}\CC_1^\op\boxtimes_{\CC_2}\CM_1)$
defines a $\FD$-$\FD$-bimodule equivalence $\FD \simeq \FM^\op\boxtimes_\FC\FM$. See the bottom figure in Figure \ref{fig:bimod-equiv2}.
\end{proof}

\begin{exam} \label{exam:double}
$(1)$ If $\CC$ is an IUMFC, $(\CC,\CC)$ is an invertible $\bk$-$(\CC,Z(\CC))$-bimodule.
$(2)$ If $\CC$ is a UMTC, $(\CC,\CC)$ is an invertible $(\CC\boxtimes\bar\CC,\bar\CC\boxtimes\CC)$-$(\CC,\bk)$-bimodule.
%$(3)$ Let $\FC$ be an enriched IUMFC. We have $(\CC_1\boxtimes\CC_1^\rev)\boxtimes_{\CC_2\boxtimes\bar\CC_2}\CC_2 \simeq \CC_1\boxtimes_{\CC_2}\CC_1^\rev$. So, $(\CC_1\boxtimes_{\CC_2}\CC_1^\rev,\CC_2)$ is an invertible $\FC\boxtimes\FC^\rev$-$(\CC_1\boxtimes_{\CC_2}\CC_1^\rev,\bk)$-bimodule.
\end{exam}

\section{A symmetric monoidal (4,3)-category}

We have a symmetric monoidal (4,3)-category $\mathbf{IUMFC}^{\mathrm{en}}$ constructed as follows. An object is an enriched IUMFC $\FC$. A morphism $\FC\leftarrow\FD$ is a $\FC$-$\FD$-bimodule. A 2-morphism is a bimodule functor. A 3-morphism is a natural transformation of bimodule functors. A 4-isomorphism is a modification. The tensor unit is $\bk$ and the tensor product is Deligne tensor product $\boxtimes$.

\begin{thm} \label{thm:dual}
The symmetric monoidal (4,3)-category $\mathbf{IUMFC}^{\mathrm{en}}$ has duals.
\end{thm}

\begin{proof}
We need to show that every object, morphism, 2-morphism has left dual and right dual.

$(1)$ The dual of an object $\FC$ is $\FC^\rev$. The unit map $\FC^\rev\boxtimes\FC \leftarrow \bk$ and the counit map $\bk \leftarrow \FC\boxtimes\FC^\rev$ are given by the trivial $\FC$-$\FC$-bimodule $\FC$ (see Remark \ref{rem:dual-obj}).

\begin{figure}[h]
\begin{tikzpicture}[scale=.28,line width=1.2pt]
\fill[lightgray] (0,0) rectangle (12,9);
\draw (3,5) node{$\CD_2$};
\draw (9,5) node{$\CD_2$};
\draw[dashed] (6,9) -- (6,0);
\draw[dashed] (6,9) [->] -- (6,3) node[right]{$\CD_2$};
\draw (0,0) -- (12,0);
\draw (0,0) [->] -- (3,0) node[below]{$\CD_1$};
\draw (6,0) node{$\bullet$} node[below]{$\CD_1$};
\draw (6,0) [->] -- (9,0) node[below]{$\CD_1$};
\draw (17,5) node{$\xrightarrow{\quad u\quad}$};
\end{tikzpicture}
\quad\quad
\begin{tikzpicture}[scale=.28,line width=1.2pt]
\fill[lightgray] (0,0) rectangle (12,9);
\draw (2.5,6) node{$\CD_2$};
%\draw (9,5) node{$\CD_2$};
\draw[dashed] (4,0) -- (6,4);
\draw[dashed] (4,0) [->] -- (5,2) node[left]{$\CD_2$};
\draw[dashed] (6,4) -- (8,0);
\draw[dashed] (6,4) [->] -- (7,2) node[right]{$\CD_2$};
\draw[dashed] (6,9) -- (6,4);
\draw[dashed] (6,9) [->] -- (6,6.5) node[right]{$\CD_2$};
\draw (6,4) node{$\bullet$} node[right]{$\CD_2$};
\draw (0,0) -- (12,0);
\draw (0,0) [->] -- (1.5,0) node[below]{$\CD_1$};
\draw (4,0) node{$\bullet$} node[below]{$\CM_0^\op$};
\draw (0,0) [->] -- (6.3,0) node[below]{$\CC_1$};
\draw (8,0) node{$\bullet$} node[below]{$\CM_0$};
\draw (6,0) [->] -- (10.5,0) node[below]{$\CD_1$};
\end{tikzpicture}
\\~\\
\begin{tikzpicture}[scale=.28,line width=1.2pt]
\fill[lightgray] (0,0) rectangle (12,9);
\draw (1.5,5) node{$\CD_2$};
\draw (6,5) node{$\CD_2$};
\draw (10.5,5) node{$\CD_2$};
\draw[dashed] (4,9) -- (4,0);
\draw[dashed] (4,9) [->] -- (4,3) node[right]{$\CD_2$};
\draw[dashed] (8,0) -- (8,9);
\draw[dashed] (8,0) [->] -- (8,3) node[right]{$\CD_2$};
\draw (0,0) -- (12,0);
\draw (0,0) [->] -- (1.5,0) node[below]{$\CC_1$};
\draw (4,0) node{$\bullet$} node[below]{$\CM_0$};
\draw (0,0) [->] -- (6,0) node[below]{$\CD_1$};
\draw (8,0) node{$\bullet$} node[below]{$\;\CM_0^\op$};
\draw (6,0) [->] -- (10.5,0) node[below]{$\CC_1$};
\draw (17,5) node{$\xrightarrow{\quad v\quad}$};
\end{tikzpicture}
\quad\quad
\begin{tikzpicture}[scale=.28,line width=1.2pt]
\fill[lightgray] (0,0) rectangle (12,9);
\draw (2.5,3) node{$\CD_2$};
%\draw (9,5) node{$\CD_2$};
\draw[dashed] (4,9) -- (6,4);
\draw[dashed] (4,9) [->] -- (5,6.5) node[left]{$\CD_2$};
\draw[dashed] (6,4) -- (8,9);
\draw[dashed] (6,4) [->] -- (7,6.5) node[right]{$\CD_2$};
\draw[dashed] (6,4) -- (6,0);
\draw[dashed] (6,4) [->] -- (6,2) node[right]{$\CD_2$};
\draw (6,4) node{$\bullet$} node[right]{$\CD_2$};
\draw (0,0) -- (12,0);
\draw (0,0) [->] -- (3,0) node[below]{$\CC_1$};
\draw (6,0) node{$\bullet$} node[below]{$\CC_1$};
\draw (6,0) [->] -- (9,0) node[below]{$\CC_1$};
\end{tikzpicture}
\caption{These two figures depict the unit map $(u,\CD_2):\FD\to\FM^\op\boxtimes_\FC\FM$ and the counit map $(v,\CD_2):\FM\boxtimes_\FD\FM^\op\to\FC$ that exhibit the duality between the  bimodules $\FM=(\CM_0,\CD_2)$ and $\FM^\op=(\CM_0^\op,\CD_2^\rev)$.}
\label{fig:dual-bimod}
\end{figure}
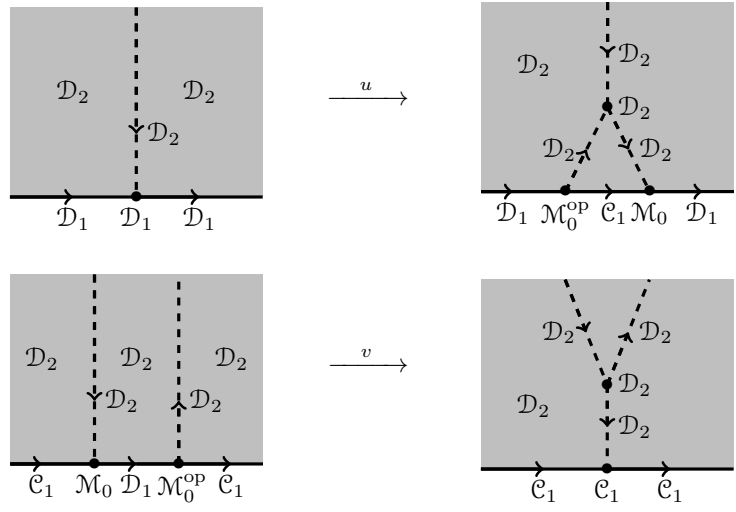

$(2)$ Let $\FM:\FC\leftarrow\FD$ be a morphism. We claim that $\FM^\op:\FD\leftarrow\FC$ is right dual to $\FM$. Indeed, in the special case where $\CC_2=\CM_1=\CD_2$, it suffices to take $(u,\CD_2):\FD\to\FM^\op\boxtimes_\FC\FM$ as unit map and take $(v,\CD_2):\FM\boxtimes_\FD\FM^\op\to\FC$ as counit map to exhibit the duality (see Figure \ref{fig:dual-bimod}), where $u:\CD_1\to\CM_0^\op\boxtimes_{\CC_1}\CM_0$ and $v:\CM_0\boxtimes_{\CD_1}\CM_0^\op\to\CC_1$ are the ordinary bimodule functors given in Theorem \ref{thm:dual-mod-cat}.
For the general case, we have a commutative diagram:
$$
\xymatrix@!C=10ex{
  & \FC':=(\CC_1\boxtimes_{\CC_2}\CM_1,\CD_2) \ar[ld]_{\FM'':=(\CC_1\boxtimes_{\CC_2}\CM_1,\CM_1)} \\
  \FC=(\CC_1,\CC_2) && \ar[ll]_-{\FM=(\CM_0,\CM_1)} \ar[lu]_{\FM':=(\CM_0,\CD_2)} \FD=(\CD_1,\CD_2). \\
}
$$
Moreover, $\FM''^\op$ is inverse to $\FM''$ by Proposition \ref{prop:inv-mor}. Replacing $\FC$ and $\FM$ by $\FC'$ and $\FM'$, respectively, we reduce the problem to the special case.
Since $(\FM^\op)^\op=\FM$, $\FM^\op$ is also left dual to $\FM$.

$(3)$ Let $\FF:\FM\to\FN$ be a 2-morphism between a pair of morphisms $\FM,\FN:\FC\leftarrow\FD$, which by definition is a left
%$\CC_1\boxtimes_{\CC_2}\CM_1\boxtimes_{\CD_2}\CD_1^\rev$-
module functor $F: \CM_0 \to \CF\boxtimes_{\CN_1}\CN_0$. The right dual of $\FF$ is the 2-morphism $\FG:\FN\to\FM$ given by the left module functor
$$G: \CN_0 \simeq \CF^\op\boxtimes_{\CM_1}\CF\boxtimes_{\CN_1}\CN_0 \xrightarrow{\Id_{\CF^\op}\boxtimes_{\CM_1}F^*} \CF^\op\boxtimes_{\CM_1}\CM_0
$$
where $F^*$ is the right adjoint functor to $F$. The unit map $\Id_\FM \to\FG\circ\FF$ and the counit map $\FF\circ\FG\to\Id_\FN$ are induced by the unit map $\Id \to F^*\circ F$ and the counit map $F\circ F^*\to\Id$, respectively. Similarly, $\FF$ has a left dual.
\end{proof}

\begin{cor}
A morphism $\FM:\FC\leftarrow\FD$ is an equivalence if and only if the $\CC_1\boxtimes_{\CC_2}\CM_1$-$\CD_1$-bimodule $\CM_0$ is invertible. In this case, $\FM^\op$ is inverse to $\FM$.
\end{cor}

\begin{proof}
This is clear from Part $(2)$ of the proof of Theorem \ref{thm:dual}.
\end{proof}

The following corollary generalizes a similar result for unitary fusion categories \cite{Mu1,eno2008}.

\begin{cor}
Two enriched IUMFC's $\FC,\FD$ are equivalent in $\mathbf{IUMFC}^{\mathrm{en}}$ if and only if $Z(\FC^\sharp) \simeq Z(\FD^\sharp)$ as UMTC's, where $Z(\FC^\sharp),Z(\FD^\sharp)$ are the centralizers of $\CC_2,\CD_2$ in $Z(\CC_1),Z(\CD_1)$, respectively (see Remark \ref{rem:enrich}).
\end{cor}

\begin{proof}
Suppose there is an equivalence $\FM:\FC\leftarrow\FD$. Note that $\CC_1$ is a closed multi-fusion $\overline{Z(\FC^\sharp)}$-$\CC_2$-bimodule, thus $Z(\CC_1\boxtimes_{\CC_2}\CM_1) \simeq Z(\FC^\sharp)\boxtimes\CD_2$. The invertible $\CC_1\boxtimes_{\CC_2}\CM_1$-$\CD_1$-bimodule $\CM_0$ induces an equivalence $Z(\FC^\sharp)\boxtimes\CD_2 \simeq Z(\CD_1)$ which preserves $\CD_2$. Thus $Z(\FC^\sharp) \simeq Z(\FD^\sharp)$.

Conversely, suppose $Z(\FC^\sharp) \simeq Z(\FD^\sharp)$. Let $\CM_1 = \CC_1^\rev\boxtimes_{\overline{Z(\FC^\sharp)}}\CD_1$. Note that $\CM_1$ is a closed multi-fusion $\CC_2$-$\CD_2$-bimodule and $Z(\CC_1\boxtimes_{\CC_2}\CM_1) \simeq Z(\FC^\sharp)\boxtimes\CD_2 \simeq Z(\CD_1)$. Therefore, there exists an invertible $\CC_1\boxtimes_{\CC_2}\CM_1$-$\CD_1$-bimodule $\CM_0$ such that the pair $(\CM_0,\CM_1)$ defines an equivalence $\FC\simeq\FD$.
\end{proof}

\begin{rem}
Theorem \ref{thm:dual} and its corollaries can be formulated and proved in terms of 1-category theory without referring to the higher category $\mathbf{IUMFC}^{\mathrm{en}}$.
\end{rem}

According to the cobordism hypothesis \cite{BD,Lu}, every object $\FC$ of $\mathbf{IUMFC}^{\mathrm{en}}$ gives rise to an extended framed 3D TQFT $\CZ_\FC$. We would like to examine its values on closed manifolds.

The value $\CZ_\FC(S^1)$ on a circle
%(equipped with the framing induced from the trivial framing of the 2-disk)
is the $\bk$-$\bk$-bimodule
\begin{align*}
\FC^\op\boxtimes_{\FC\boxtimes\FC^\rev}\FC
& = (\CC_1^\op,\CC_2^\rev)\boxtimes_{(\CC_1,\CC_2)\boxtimes(\CC_1^\rev,\bar\CC_2)}(\CC_1,\CC_2) \\
& = (\CC_1^\op\boxtimes_{\CC_1\boxtimes\CC_1^\rev}\CC_1,\CC_2^\rev\boxtimes_{\CC_2\boxtimes\bar\CC_2}\CC_2) \\
& \simeq (\Fun_{\CC_1\boxtimes\CC_1^\rev}(\CC_1,\CC_1),\Fun(\CC_2,\CC_2)) \\
& \simeq (Z(\CC_1),\Fun(\CC_2,\CC_2)) \\
& \simeq (Z(\FC^\sharp),\bk)\boxtimes(\CC_2,\Fun(\CC_2,\CC_2)) \\
& \simeq (Z(\FC^\sharp),\bk)
\end{align*}
where the first $\simeq$ is due to \eqref{eq:xRy} and Corollary \ref{cor:inv-bimod}(1), and the last $\simeq$ is due to Example \ref{exam:trivial}. This result agrees with the
%is compatible with the fact that the Drinfeld center $Z(\FC^\sharp)$ is equivalent to $\CC$ (see Remark \ref{rem:enrich}), as it is a
general phenomenon that the value of an extended TQFT on a circle is the ``center" of that on a point.

The value $\CZ_\FC(\Sigma)$ on a closed surface $\Sigma$ is an $\bk$-$\bk$-bimodule functor $\bk\to\bk$. By definition, it is encoded by a pair $(u_\Sigma,\bk_\Sigma)$ where $\bk_\Sigma$ is a unitary category equivalent to $\bk$ and $u_\Sigma:\bk\to \bk_\Sigma$ is a functor or, equivalently, an object of $\bk_\Sigma$.

The value $\CZ_\FC(M)$ on a closed 3-manifold $M$ is a natural transformation $\Id_\bk\to\Id_\bk$ of $\bk$-$\bk$-bimodule functors. By definition, it is encoded by a pair $(v_M,\C_M)$ where $\C_M:\bk\to\bk$ is an equivalence or, equivalently, a one-dimensional Hilbert space and $v_M:\Id_\bk\to \C_M$ is a natural transformation or, equivalently, a vector in $\C_M$.

\smallskip

Note that, a unitary fusion category $\CC$ can be regarded as an enriched IUMFC $(\CC,\bk)$. So, $\mathbf{IUMFC}^{\mathrm{en}}$ contains the symmetric monoidal 3-category formed by unitary fusion categories, their ordinary bimodules, bimodule functors and natural transformations. Moreover, this symmetric monoidal 3-subcategory is closed under duality. Therefore, $\CZ_{(\CC,\bk)}$ is nothing but the extended TQFT given in \cite{dss1} which is believed to be an extension of the Turaev-Viro invariant associated to $\CC$. In another word, the extended Turaev-Viro TQFT $\CZ^{TV}_\CC$ is included here as a special case.

%It would be also interesting to observe the double theory $\CZ_{\FC\boxtimes\FC^\rev}$. According to Example \ref{exam:double}, we have an equivalence $\FC\boxtimes\FC^\rev \simeq (\CC_1\boxtimes_{\CC_2}\CC_1^\rev,\bk)$ as objects of $\mathbf{IUMFC}^{\mathrm{en}}$. Therefore, the double theory $\CZ_{\FC\boxtimes\FC^\rev}$ is equivalent to the extended Turaev-Viro TQFT associated to the IUMFC $\CC_1\boxtimes_{\CC_2}\CC_1^\rev$.

\smallskip

Now we focus on the special case $\FC = (\CC,\bar\CC)$ where $\CC$ is a UMTC. On the one hand side, we have an extended framed 3D TQFT $\CZ_\CC$ given by the fully dualizable object $\FC$.
On the other hand side, there is a 1-2-3-dimensional Reshetikhin-Turaev TQFT $\CZ^{RT}_\CC$ associated to the UMTC $\CC$ \cite{Tu,BK} as well as an invertible extended 4D TQFT $\CA_\CC$ associated to the same UMTC $\CC$ \cite{Fr3}.

We conjecture that the 1-2-3-dimensional theory of $\CZ_\CC$ is a combination of those of $\CZ^{RT}_\CC$ and $\CA_\CC$. In particular, this provides a way to extend the Reshetikhin-Turaev TQFT down to dimensional zero.
In fact, as we have seen, the values of $\CZ_\CC$ on these dimensions are always encoded by pairs, and the examples we have examined match perfectly with $\CZ^{RT}_\CC$ and $\CA_\CC$. That is, the values $\Fun(\CC,\CC),\bk_\Sigma,\C_M$ have the same form as those of $\CA_\CC$ and the values $Z(\FC^\sharp)\simeq\CC,u_\Sigma,v_M$ have the same form as those of $\CZ^{RT}_\CC$.

Another evidence of this conjecture arises from the double theory $\CZ_{\CC\boxtimes\bar\CC}$. By Example \ref{exam:double}(2), $(\CC\boxtimes\bar\CC,\bar\CC\boxtimes\CC) \simeq (\CC,\bk)$ as objects of $\mathbf{IUMFC}^{\mathrm{en}}$. Therefore, the double theory $\CZ_{\CC\boxtimes\bar\CC}$ is equivalent to the extended Turaev-Viro TQFT $\CZ^{TV}_\CC$ associated to $\CC$. This also matches perfectly with the fact that the double theory $\CZ^{RT}_{\CC\boxtimes\bar\CC}$ is equivalent to $\CZ^{TV}_\CC$ \cite{Ba2} (while the double theory $\CA_{\CC\boxtimes\bar\CC}$ is trivial).

\smallskip

Moreover, we conjecture that every object of $\mathbf{IUMFC}^{\mathrm{en}}$ can be promoted canonically to a $SO(3)$-fixed point. That is, $\CZ_\FC$ defines an extended 3D TQFT without framing anomaly. There are several evidences for this conjecture. First, in the special case $\CZ_\CC$, there have already been such data as $\bk_\Sigma$, $\C_M$ to address the framing anomaly of the Reshetikhin-Turaev TQFT.

Secondly, recall that, although $\CZ_\FC(S^1)$ is equivalent to one of its factor $(Z(\FC^\sharp),\bk)$, it contains another factor $(\CC_2,\Fun(\CC_2,\CC_2)) \simeq \bk$ which does contribute to higher dimensional cobordisms and yields such data as $\bk_\Sigma$, $\C_M$. From the physical point of view, these two factors of $\CZ_\FC(S^1)$ constitute the chiral and anti-chiral parts of the whole theory so that their framing anomalies are canceled.

Thirdly, the symmetric monoidal (4,3)-category $\mathbf{IUMFC}^{\mathrm{en}}$ has a very nice behavior under duality, due to the unitarity condition we imposed in the construction. For example, the dual of a $\FC$-$\FD$-bimodule $\FM$ is given by $\FM^\op$ no matter we treat it as a $\FC$-$\FD$-bimodule or an $\bk$-$\FC^\rev\boxtimes\FD$-bimodule instead. This fact would have sufficed to promote every object of $\mathbf{IUMFC}^{\mathrm{en}}$ canonically to a $SO(2)$-fixed point. However, promoting to a $SO(3)$-fixed point is a much more subtle problem. See \cite{dss1} for a similar conjecture and discussions therein.

\begin{rem}
The unitarity is only used to eliminate the ambiguity of left/right dual of an object in a multi-fusion category. All results are equally true if we replace UMFC's and UMTC's by spherical multi-fusion categories and modular tensor categories. It is also possible to generalize these results to multi-fusion categories and nondegenerate braided fusion categories by taking care of the difference between left/right dual.
\end{rem}

%\begin{rem}
%If our conjecture is true, extended Reshetikhin-Turaev TQFT and extended Turaev-Viro TQFT are incorporated into a common framework. One might go one step further to incorporate the extended Crane-Yetter TQFT. For this, one follows the line of \cite{Fr3} and considers a symmetric monoidal $(5,4)$-category with duals whose object is a UMTC $\CC$. A morphism $\CC\leftarrow\CD$ is an enriched IUMFC $\FC$ equipped with a braided monoidal functor $\bar\CC\boxtimes\CD \to Z(\FC^\sharp)$, and higher morphisms are given by bimodules, bimodule functors, and so forth. Then a UMTC $\CC$ gives rise to an extended 4D TQFT, which is nothing but the extended Crane-Yetter TQFT $\CA_\CC$. Moreover, the enriched IUMFC $(\CC,\bar\CC)$ defines an equivalence between $\bk$ and $\CC$ thus, according to the cobordism hypothesis, gives rise to an equivalence from the trivial TQFT $\CA_\bk$ to $\CA_\CC$. This realizes $\CZ_\CC$ as a boundary theory of $\CA_\CC$.
%\end{rem}

\begin{rem}
In a joint work with Liang Kong, we will explain how to use enriched fusion categories to describe both gapped edges and gapless edges of 2+1D topological orders. Roughly speaking, a gapped edge supports a topological field theory which is described mathematically by a unitary fusion category or, equivalently, a unitary fusion category enriched in $\bk$. However, a gapless edge supports a conformal field theory which is described mathematically by a unitary fusion category enriched in the module category of a unitary vertex operator algebra. 

This also shed light on Chern-Simons theory with compact gauge group which is used to study the effective theory of 2+1D topological orders. If the gauge group is finite, the boundary theory is topological hence is described by a unitary fusion category as it was done in Freed's work \cite{Fr2} (see also \cite{FHLT}). If the gauge group is not finite, the boundary theory is a conformal field theory hence might be described by an enriched unitary fusion category. This idea is very close to that in Henriques' work \cite{He1,He2}. It would be very interesting to clarify the relation between  categories of solitons arising from conformal net and enriched fusion categories.
\end{rem}

\end{document}